\documentclass[reqno]{amsart}
\usepackage{amsthm}
\usepackage{amsfonts}
\usepackage{amssymb}
\usepackage{amsmath}
\usepackage{mathrsfs}
\usepackage{csquotes}

\begin{document}
	%%%%%%%%%%%%%%%%%%%%%%%%%%%%%%%%%%%%%%%%%%%%%%%%%%%%%%%%%%%%%%%%%%%%%%%%%%%%
	\theoremstyle{plain}
	\theoremstyle{definition}
	
	\newtheorem{theorem}{Theorem}[section]
	\newtheorem{problem}[theorem]{Problem}
	\newtheorem{proposition}[theorem]{Proposition}
	\newtheorem{corollary}[theorem]{Corollary}
	\newtheorem{lemma}[theorem]{Lemma}
	\newtheorem{example}[theorem]{Example}
	\newtheorem{remark}[theorem]{Remark}
	\newtheorem{definition}[theorem]{Definition}
	\newtheorem{fact}[theorem]{Fact}
	\newtheorem{claim}[theorem]{Claim}
	\newtheorem{note}{\textbf{Note}}[section]
	
	\large
	\title[On the role of the $Ky~Fan$ metric in rough ideal convergence in probability]{On the Role of the $Ky~Fan$ Metric in Rough Ideal Convergence in Probability}

	
	
	\date{}
		
	\author[Tamim Aziz and Sanjoy Ghosal]{Tamim Aziz and Sanjoy Ghosal}
	\thanks{Department of Mathematics, University of North Bengal, Raja Rammohunpur, Darjeeling-734013, West Bengal, India. \emph{E-mail}: tamimaziz99@gmail.com (T. Aziz); sanjoykumarghosal@nbu.ac.in; sanjoyghosalju@gmail.com (S. Ghosal). \\ Orcid: https://orcid.org/0009-0004-5727-0121 (T. Aziz);  https://orcid.org/0000-0001-8563-1941 (S. Ghosal)\\	The research of the first author is supported by Human Resource Development Group, CSIR, India, through NET-JRF and the grant number is 09/0285(12636)/2021-EMR-I}

		\setcounter {page}{1}

	\begin{abstract}
		Given a probability space $(S,\Delta, \mathbb{P})$ and a separable metric space $(U,d)$, the $Ky~Fan$ metric $\rho(X,Y)$ on the space $\mathfrak{X}^0$ of equivalence classes of random variables (w.r.t. almost sure equality) formed from the set $\mathfrak{X}(U)$ of $U$-valued random variables is given by $
		\rho(X,Y)=\inf \{\varepsilon>0:\mathbb{P}(d(X,Y)>\varepsilon)\leq\varepsilon\}.$ In this article, we primarily introduce the concept of rough ideal convergence in probability which serves as a unifying generalization of both ideal convergence of sequences in metric spaces and convergence of random variables in probability.
		We demonstrate that the rough ideal limit set is closed and bounded w.r.t. the $Ky~Fan$ metric $\rho$, and that, for a certain class of ideals, it forms an $F_{\sigma\delta}$ subset of $\mathfrak{X}^0$. In this process, we present the key concepts of strong and weak rough ideal cluster points in probability. It turns out that the set of strong rough ideal cluster points in probability is always closed, whereas the weak set is conditionally closed in the metric space ($\mathfrak{X}^0,\rho)$.  Finally, we obtain a characterization of a maximal admissible ideal in terms of the sets of strong rough ideal cluster points and the rough ideal limit set in probability.
	\end{abstract}

	\maketitle

	\section{Introduction}
	
	Throughout this article,  $(S,\Delta, \mathbb{P})$ is a probability space, $(U,d)$ is a separable metric space, and $\mathfrak{X}:=\mathfrak{X}(U)$ is the set of $U$-valued random variables on $S$. For each $X,Y\in \mathfrak{X}$, we consider the $Ky$ $Fan$ metric $\rho$ on $\mathfrak{X}$ is defined by $\rho(X,Y)=\inf \{\varepsilon>0:\mathbb{P}(d(X,Y)>\varepsilon)\leq\varepsilon\}$. Note that $(\mathfrak{X^0},\rho)$ is a metric space, where $\mathfrak{X^0}$ represents the set of equivalence classes of random variables in $\mathfrak{X}$, and two random variables are considered equivalent if they are equal almost surely. The primary motivation for restricting attention to separable metric spaces $(U,d)$, rather than considering arbitrary ones, lies in the necessity of the measurability of the metric $d$. This measurability is crucial for establishing a connection between the metric structure of $U$ and that of $\mathfrak{X}(U)$. In particular, the measurability of $d$ enables us to handle the probability metric $\rho$ in a well-defined manner (see \cite{rachev2013methods} for further reading).\\

Let us begin by recalling the concept of a submeasure on $\mathbb{N}$ that plays an  important role in this paper. A map $\varphi:\mathcal{P}(\mathbb{N})\to [0,\infty]$ is a submeasure on $\mathbb{N}$ if (i) $\varphi(\varnothing)=0,$ (ii)	if $A\subseteq B,$ then $\varphi(A)\leq \varphi(B),$ (iii)	 $\varphi(A\cup B)\leq \varphi(A)+\varphi(B)$ for any $A,B\subseteq\mathbb{N}$, and (iv) $\varphi(\{t\})<\infty,$ for all $t\in\mathbb{N}.$ A  submeasure  $\varphi$ is  lower semicontinuous   (briefly, lscsm) if $\varphi(A)=\displaystyle\lim_{t\to\infty}\varphi(A\cap[1,t])$ for all $A\subseteq\mathbb{N}.$
For each lscsm $\varphi$ on $\mathbb{N},$ the exhaustive ideal $Exh(\varphi),$ generated by $\varphi,$ is defined as follows: 
\begin{equation*}
	Exh(\varphi)=\{A\subset \mathbb{N}:\lim_{t\to\infty}\varphi(A\setminus\{1,2,...,t\})=0\}.
\end{equation*}
At this stage, we consider an important category of set-theoretical objects, namely ``ideals".
A family $\mathcal{I}\subset\mathcal{P}(\mathbb{N})$ is referred to as an ideal \cite{kostyrko2000convergence} on $\mathbb{N}$ if it fulfills the following conditions:
\begin{itemize}
	\item $\varnothing\in\mathcal{I}$,
	\item if $A,B\in \mathcal{I}$ then $A\cup B\in\mathcal{I}$,
	\item if $A\subset B$ and $B\in\mathcal{I}$ then $A\in \mathcal{I}$.
\end{itemize}
An ideal $\mathcal{I}$ is called non-trivial if $\mathcal{I}\neq\varnothing$ and $\mathcal{I}\neq\mathbb{N}$. We denote by $\mathcal{I}_{fin}$ the ideal of finite subsets of $\mathbb{N}$ and by $\mathcal{I}_\delta=\{A\subseteq \mathbb{N}: \delta(A)=0\}$ the ideal of subsets of $\mathbb{N}$ with natural density zero, where $\delta$ denotes the natural density  \cite{fast1951convergence, fridy1985statistical,fridy1993statistical, vsalat1980statistically,steinhaus1951convergence}. A non-trivial ideal $\mathcal{I}$ is considered admissible if $\{t\}\in\mathcal{I}$ for each $t\in\mathbb{N}$, i.e., $\mathcal{I}\supseteq \mathcal{I}_{fin}$. %and it is maximal if for each $A\subset\mathbb{N}$ either $A\in\mathcal{I}$ or $\mathbb{N}\setminus A\in\mathcal{I}$. 
Furthermore, an ideal $\mathcal{I}$ is a $P$-ideal \cite{kostyrko2000convergence} if for every sequence $\{A_t\}_{t\in\mathbb{N}}$ of sets in $\mathcal{I}$ there exists  $A\in\mathcal{I}$ such that  $A_t\setminus A$ is finite for each $t\in\mathbb{N}.$ Examples of $P$-ideals include $\mathcal{I}_{fin}$ and $\mathcal{I}_\delta$. A $P$-ideal $\mathcal{I}$ is termed analytic if there exists a lscsm $\varphi$ on $\mathbb{N}$ such that $\mathcal{I}=Exh(\varphi)$ (visit \cite{solecki}). For an ideal $\mathcal{I}$, we will write $\mathcal{F}(\mathcal{I})=\{A\subseteq\mathbb{N}:\mathbb{N}\setminus A\in\mathcal{I}\}$ to denote the dual filter of $\mathcal{I}$.\\

In a different direction, Phu \cite{phu2001rough,xuan2003rough} began an exploration of the theory of rough convergence of a sequence, which serves as an extension of classical convergence in normed spaces, where ``degree of roughness" is recognized as a crucial factor. He also presented several intriguing properties of the set of rough limit points of a sequence in normed spaces, which were quite fascinating and pertinent to this study. Before we proceed, let us formally present the concept of \textit{rough convergence} of a sequence.
\begin{definition}\cite[Page no. 199]{phu2001rough} \cite[Eq (1.1)]{xuan2003rough} Let $r$ be a non-negative real number. A sequence  $\{x_n\}_{n\in \mathbb{N}}$ in a normed space $X$ is said to be  rough convergent  to $x_*$ w.r.t. the degree of roughness $r$ (briefly, $r$-convergent), denoted by $x_n\xrightarrow{r}x_*,$  provided that
	\begin{equation*}~\mbox{for any}~ \varepsilon > 0,~\mbox{there exists}~ n_{\varepsilon} \in \mathbb{N}: n\geq n_{\varepsilon} ~\Rightarrow~ \|x_n - x_*\| \leq  r+ \varepsilon.\end{equation*}
	The  non-negative real number $r$ is known as ``degree of roughness'' and the set $LIM^r x_i=\left\{x_*\in X: x_n\xrightarrow{r}x_*\right\}$ is called the $r$-limit set of the sequence $\{x_n\}_{n\in \mathbb{N}}.$
\end{definition}
Subsequently, this notion has been extended to rough ideal convergence as follows:
\begin{definition}\cite{dundar2014rough,PCD}
	A sequence  $\{x_n\}_{n\in \mathbb{N}}$ in a normed space $X$ is said to be  rough $\mathcal{I}$-convergent  to $x_*$ w.r.t. the degree of roughness $r\geq 0$ (briefly, $r-\mathcal{I}$-convergent),  provided that
	\begin{equation*}
		\{n\in\mathbb{N}:\|x_n - x_*\|>  r+ \varepsilon\}\in\mathcal{I}
		~\mbox{ for every }~ \varepsilon > 0.\end{equation*}
\end{definition}
For a comprehensive overview of established results on rough convergence, along with relevant references, visit \cite{aytarcore,aytar2008rough,aytar2017rough,listan2011characterization,rahaman2024rough}.\\

 On the other hand, convergence of random variables in probability is a central concept in probability theory, expressing the concept that a sequence of random variables converges toward a specific random variable with probability becoming arbitrarily close to $1$. Formally,
\begin{definition}\cite{rachev2013methods,resnick2003probability}
	A sequence $\{X_n\}_{n\in\mathbb{N}}$ of $U$-valued random variables, defined on a sample space $S$, 
	is said to converge in probability to a $U$-valued random variable $Y$ if, for any 
	$\varepsilon>0$, the probability that $d(X_n,Y)>\varepsilon$ tends to $0$ as $n\to\infty$.
\end{definition}
Over the years, several generalizations of this concept have been proposed, notably by \cite{ghosal2013statistical,jena2020various, rambaud2011note,csenccimen2013statistical,csenccimen2008strong,csenccimen2016exhaustive}, which extend the classical framework to encompass more general settings and refined modes of convergence.\\

We are now ready to present our main definition, namely Definition \ref{defmain}, that generalizes the notion of convergence in probability for sequences of random variables.
	\begin{definition}\label{defmain}
		Assume that $r\geq 0$ and $\{X_n\}_{n\in\mathbb{N}}$ is a sequence in $\mathfrak{X}$. Then $\{X_n\}_{n\in\mathbb{N}}$ is said to be rough $\mathcal{I}$-convergent in probability to $X_*\in\mathfrak{X}$ w.r.t. the degree of roughness $r$ (briefly, $r-\mathcal{I}^{\mathbb{P}}$ convergence), denoted by $X_n\xrightarrow[r]{\mathcal{I}^{\mathbb{P}}} X_*$, provided that
		$$\{n\in\mathbb{N}:\mathbb{P}(d(X_n,X_*)> r+\varepsilon)>\delta\}\in\mathcal{I}~\mbox{ for every $\varepsilon,\delta > 0$}.$$ 
		We denote $\mathcal{I}^{\mathbb{P}}$-$LIM^rX_i=\left\{X_*\in \mathfrak{X}:X_n\xrightarrow[r]{\mathcal{I}^{\mathbb{P}}} X_* \right\}$ the set of all $r$-$\mathcal{I}^{\mathbb{P}}$ limits of $\{X_n\}_{n\in\mathbb{N}}$.  Furthermore, when $r=0$, we call $\{X_n\}_{n\in\mathbb{N}}$ to be $\mathcal{I}$-convergent in probability (briefly, $\mathcal{I}^{\mathbb{P}}$-convergent) to $X_*\in\mathfrak{X}$. Additionally, by replacing the ideal $\mathcal{I}$ by $\mathcal{I}_{fin}$ and $\mathcal{I}_{\delta}$, one obtains the notions of convergence in probability \cite{resnick2003probability} and statistical convergence in probability \cite{ghosal2013statistical}, respectively.
	\end{definition}
	
	\begin{remark}\label{impremark}
		Let $\{x_n\}_{n\in\mathbb{N}}$ be a sequence in $(U,d)$. For each $n\in\mathbb{N}$, we define $X_n\in\mathfrak{X}$ such that $\mathbb{P}(X_n=x_n)=1$ and $\mathbb{P}(X_n\neq x_n)=0$.
		 Note that, for each $n\in\mathbb{N}$, $X_n$ has one-point distribution at $x_n$. We call $\{X_n\}_{n\in\mathbb{N}}$ the associated sequence of random variables of the sequence $\{x_n\}_{n\in\mathbb{N}}$.
	\end{remark}
	The following result motivates the investigation of this new notion of rough ideal convergence in probability, as it establishes that rough ideal convergence of random variables in probability is a more generalized form of rough ideal convergence of sequences in metric spaces.
	\begin{proposition}
		Let $\{x_n\}_{n\in\mathbb{N}}$ be a sequence in $(U,d)$. If there exists $r\geq 0$ such that $\{x_n\}_{n\in\mathbb{N}}$ is $r$-$\mathcal{I}$ convergent to $x_*\in U$ then the associated sequence $\{X_n\}_{n\in\mathbb{N}}$ of random variables  satisfies $X_n\xrightarrow[r]{\mathcal{I}^{\mathbb{P}}} X_*$, where $\mathbb{P}(X_*=x_*)=1$.
	\end{proposition}
	\begin{proof}
	Assume that $\{x_n\}_{n\in\mathbb{N}}$ is $r$-$\mathcal{I}$ convergent to $x_*$. Let $\varepsilon>0$ be given. Therefore, we have
	\begin{equation}\label{eq 1}
		A(\varepsilon)=\{n\in\mathbb{N}:d(x_n,x_*)>r+\varepsilon\}\in\mathcal{I}.
	\end{equation}
	Note that $\mathbb{P}(A_n)=1$ and $\mathbb{P}(A)=1$, where 
	$$A_n=\{\omega\in S: X_n(\omega)=x_n\}~\mbox{ and }~A=\{\omega\in S: X_*(\omega)=x_*\}.$$
 Observe that $\mathbb{N}\setminus A(\varepsilon)\neq\varnothing$ since $A(\varepsilon)\in \mathcal{I}$. So pick arbitrary $n\in \mathbb{N}\setminus A(\varepsilon)$. Then, in view of Eq (\ref{eq 1}), one obtains $\mathbb{P}(d(X_n,X_*)\leq r+\varepsilon)=\mathbb{P}(A_n\cap A)=1$. Therefore, for any $\delta>0$, we can write 
	$$\{n\in\mathbb{N}:\mathbb{P}(d(X_n,X_*)> r+\varepsilon)>\delta\}\subseteq A(\varepsilon).$$
	As a consequence, we obtain that $X_n\xrightarrow[r]{\mathcal{I}^{\mathbb{P}}} X_*$.
	\end{proof}
	
%	The converse of the above result is not true in general.
%	\begin{example}
%	Consider the metric space $(\mathbb{R},d_u)$, where $d_u(x,y)=|x-y|$ for all $x,y\in\mathbb{R}$.
%	\end{example}
It is well-known folklore result that the $Ky$ $Fan$ metric $\rho$ metrizes converges in probability. This naturally leads to the question of whether a corresponding generalization holds within the framework of ideal convergence. The following result provides an affirmative answer.
\begin{proposition}\label{pro 1}
Let $\{X_n\}_{n\in\mathbb{N}}$ be a sequence of random variables taking values in the metric space $(U,d)$. Then $\{X_n\}_{n\in\mathbb{N}}$ is $\mathcal{I}$-convergent in probability to $X\in \mathfrak{X}$ if and only if it is $\mathcal{I}$-convergent to $X$ via the $Ky$ $Fan$ metric $\rho$.
\end{proposition}
\begin{proof}
First assume that $\{X_n\}_{n\in\mathbb{N}}$ is $\mathcal{I}$-convergent in probability to $X\in \mathfrak{X}$. Let $\varepsilon>0$ be given. Then 
$$
A(\varepsilon)=\{n\in\mathbb{N}:\mathbb{P}(d(X_n,X_*)\geq \varepsilon)> \varepsilon\}\in \mathcal{I}.
$$
Let us set $B(\varepsilon)=\{n\in\mathbb{N}:\rho(X_n,X_*)\geq \varepsilon\}$. Note that, for any $n\notin A(\varepsilon)$, we have $\mathbb{P}(d(X_n,X_*)\geq \varepsilon)\leq\varepsilon$. This ensures that $\rho(X_n,X_*)<\varepsilon$, i.e.,  $n\notin B(\varepsilon)$. Thus we obtain  $B(\varepsilon)\subseteq A(\varepsilon)$. Since  $A(\varepsilon)\in\mathcal{I}$ and $\varepsilon>0$ was arbitrary, we conclude that $\{X_n\}_{n\in\mathbb{N}}$ is $\mathcal{I}$-convergent to $X$ via $\rho$.\\

Next, let us assume that $\{X_n\}_{n\in\mathbb{N}}$ is $\mathcal{I}$-convergent to $X$ via $\rho$. Let $\varepsilon,\delta>0$ and set $\rho_n=\rho(X_n,X)$. Then, for each $\eta>0$, we have $A(\eta):=\{n\in\mathbb{N}:\rho_n\geq \eta\}\in\mathcal{I}$. Now pick any $j\in \mathbb{N}\setminus (A(\varepsilon)\cup A(\delta))$ (note that such $j\in\mathbb{N}$ exists since $A(\varepsilon)\cup A(\delta)\in \mathcal{I}$). Therefore, observe that 
\begin{align*}
	\mathbb{P}(d(X_j,X)>\varepsilon)&\leq \mathbb{P}(d(X_j,X)>\rho_j)\\
	&\leq \rho_j~\mbox{ (since the infimum value in the definition of $\rho(X_j,X)$ is reached)}\\
	&<\delta.
\end{align*}
Thus  $\{j\in\mathbb{N}:\mathbb{P}(d(X_j,X)>\varepsilon)\geq\delta\}\subseteq A(\varepsilon)\cup A(\delta)$, i.e., $\{j\in\mathbb{N}:\mathbb{P}(d(X_j,X)>\varepsilon)\geq\delta\}\in\mathcal{I}$. Hence we deduce that  $\{X_n\}_{n\in\mathbb{N}}$ is $\mathcal{I}$-convergent in probability to $X\in \mathfrak{X}$.
\end{proof}

The discussion now shifts to the central theme of this article, with the subsequent sections organized as follows. Section \ref{secti2} is devoted to establishing various characterizations of the set $\mathcal{I}^{\mathbb{P}}$-$LIM^rX_i$ w.r.t. the $Ky$ $Fan$ metric. In particular, it is established that, for an analytic $P$-ideal $\mathcal{I}$, the set  $\mathcal{I}^{\mathbb{P}}$-$LIM^rX_i$ is a Borel set of $F_{\sigma\delta}$ type in ($\mathfrak{X}^0,\rho)$
(Theorem \ref{theo1}). Section \ref{secti3} naturally leads to the introduction of the notions of rough ideal limit and cluster points in probability. In contrast to the case of real sequences, this framework gives rise to two distinct types of cluster points in probability, namely weak rough ideal cluster points and strong rough ideal cluster points in probability. In this section, we undertake a thorough investigation of the interrelations among these three sets. In particular, we show that the set of strong rough ideal cluster points in probability is always closed, whereas the weak set is only conditionally closed in the metric space ($\mathfrak{X}^0,\rho)$ (Propositions \ref{strong closed} and \ref{weakclosed}). Finally, we show that an admissible ideal is maximal precisely when the rough ideal limit set and the set of strong rough cluster points in probability coincide (Theorem \ref{maximal th}).

	\section{Some characterizations of the set $\mathcal{I}^{\mathbb{P}}$-$LIM^rX_i$ via $Ky$ $Fan$ metric}\label{secti2}
In this section, we explore the properties of the limit set $\mathcal{I}^{\mathbb{P}}$-$LIM^rX_i$ in the $Ky$ $Fan$ metric space $(\mathfrak{X^0},\rho)$, providing a detailed analysis of its structural and topological characteristics. Our first result portrays that the diameter of $\mathcal{I}^{\mathbb{P}}$-$LIM^rX_i$ does not exceed $2r$. 
\begin{theorem}
The set $\mathcal{I}^{\mathbb{P}}$-$LIM^rX_i$ is bounded in $(\mathfrak{X^0},\rho)$ and $diam_{\rho}(\mathcal{I}^{\mathbb{P}}$-$LIM^rX_i)\leq \min\{1,2r\}$. In general, the diametric bound cannot be reduced further.
\end{theorem}
\begin{proof}
Let us pick  arbitrary $X_*,Y_*\in$$\mathcal{I}^{\mathbb{P}}$-$LIM^rX_i$. Let $\varepsilon,\delta>0$. Then we have 
$$
A(\varepsilon,\delta)=\{n\in\mathbb{N}:\mathbb{P}(d(X_n,X_*)\geq r+\frac{\varepsilon}{2})>\frac{\delta}{2}\}\in \mathcal{I},$$
$$\mbox{and }~B(\varepsilon,\delta)=\{n\in\mathbb{N}:\mathbb{P}(d(X_n,Y_*)\geq r+\frac{\varepsilon}{2})>\frac{\delta}{2}\}\in \mathcal{I}.
$$
Now fix any $j\in\mathbb{N}\setminus (A(\varepsilon,\delta)\cup B(\varepsilon,\delta))$. Therefore, we can write
\begin{align*}
	\mathbb{P}(d(X_*,Y_*)\geq 2r+\varepsilon)&\leq \mathbb{P}(d(X_j,X_*)\geq {r}+\frac{\varepsilon}{2})+\mathbb{P}(d(X_j,Y_*)\geq {r}+\frac{\varepsilon}{2})\\
	&\leq\delta.
\end{align*}
Since $\varepsilon,\delta>0$ were chosen arbitrarily, we get that $\mathbb{P}(d(X_*,Y_*)\geq 2r)=0$. Consequently, we obtain that $\rho(X_*,Y_*)\leq 2r$. Since  $\rho(X,Y)\leq 1$ for all $X,Y\in \mathfrak{X^0}$, we conclude that $diam_{\rho}(\mathcal{I}^{\mathbb{P}}$-$LIM^rX_i)\leq \min\{1,2r\}$.\\

Next, we consider the ideal $\mathcal{I}_{\frac{1}{n}}=\{A\subset\mathbb{N}:\sum_{n\in A}\frac{1}{n}<\infty\}$ and the sequence $\{X_n\}_{n\in\mathbb{N}}$ of real valued random variables such that 
	\begin{align*}
	X_{n}\in
	\begin{cases}
		\{-5,5\}~\mbox{ with }~ \mathbb{P}(X_n=-5)=\mathbb{P}(X_n=5)~&\mbox{ if }~n\in \{2^m: m\in\mathbb{N}\}, \\
	\{0,1\}~\mbox{ with }~ \mathbb{P}(X_n=0)=1-\frac{1}{n}~\mbox{ and }~ \mathbb{P}(X_n=1)=\frac{1}{n}~&\mbox{ elsewhere}.
	\end{cases}
\end{align*}
Let $r\in\mathbb{R}$ be such that  $0<2r<1$. We now consider $X_*,Y_*\in\mathfrak{X}(\mathbb{R})$  such that $$\mathbb{P}(X_*=r)=\mathbb{P}(Y_*=-r)=1.$$
 Then, for any $\varepsilon,\delta>0$, we have
 $$\{n\in\mathbb{N}:\mathbb{P}(|X_n-X_*|>r+\varepsilon)>\delta\}\subset^*\{2^m:m\in\mathbb{N}\}.$$
As a consequence, $X_*\in $$\mathcal{I}_{\frac{1}{n}}^{\mathbb{P}}$-$LIM^rX_i$. Likewise, we can infer that $Y_*\in $$\mathcal{I}_{\frac{1}{n}}^{\mathbb{P}}$-$LIM^rX_i$. Now it is easy to observe that $\rho(X_*,Y_*)=2r$. Subsequently, $diam_{\rho}(\mathcal{I}_{\frac{1}{n}}^{\mathbb{P}}$-$LIM^rX_i)=2r$. This particular instance demonstrates that the diametric bound cannot be further diminished.
\end{proof}

\let\thefootnote\relax\footnotetext{For any two subsets $A$ and $B$ of $\mathbb{N}$ we will denote $A\subset^*B$ if $A\setminus B$ is finite and $A=^*B$ if $A\Delta B$ is finite.}
A fundamental topological property of the limit set $\mathcal{I}^{\mathbb{P}}$-$LIM^rX_i$ is given by the subsequent result.
\begin{theorem}
The set  $\mathcal{I}^{\mathbb{P}}$-$LIM^rX_i$ is closed in $(\mathfrak{X^0},\rho)$.
\end{theorem}
\begin{proof}
Let $\{Y_n\}_{n\in\mathbb{N}}$ be a sequence in $\mathcal{I}^{\mathbb{P}}$-$LIM^rX_i$ such that $\{Y_n\}_{n\in\mathbb{N}}$ is $\rho$-convergent to $X_*\in \mathfrak{X^0}$. Since $\rho$ metrizes convergence in probability, for any $\varepsilon,\delta>0$ there exists $j_{\varepsilon,\delta}\in\mathbb{N}$ such that 
\begin{equation*}
\mathbb{P}(d(Y_{j_{\varepsilon,\delta}},X_*)\geq \frac{\varepsilon}{2})\leq \frac{\delta}{2}
\end{equation*}
Therefore, observe that 
\begin{align*}
	\mathbb{P}(d(X_n,X_*)>r+\varepsilon)&\leq 	\mathbb{P}(d(X_n,Y_{j_{\varepsilon,\delta}})>r+\frac{\varepsilon}{2})+\mathbb{P}(d(Y_{j_{\varepsilon,\delta}},X_*)>\frac{\varepsilon}{2})\\
	&\leq \mathbb{P}(d(X_n,Y_{j_{\varepsilon,\delta}})>r+\frac{\varepsilon}{2})+\frac{\delta}{2}.
\end{align*}
Thus we obtain that $$\{n\in\mathbb{N}:\mathbb{P}(d(X_n,X_*)>r+\varepsilon)>\delta\}\subseteq \{n\in\mathbb{N}:\mathbb{P}(d(X_n,Y_{j_{\varepsilon,\delta}})\geq r+\frac{\varepsilon} {2})>\frac{\delta}{2}\}.$$
Since $Y_{j_{\varepsilon,\delta}}\in\mathcal{I}^{\mathbb{P}}$-$LIM^rX_i$, we have $\{n\in\mathbb{N}:\mathbb{P}(d(X_n,X_*)>r+\varepsilon)>\delta\}\in \mathcal{I}$. As a consequence, $X_*\in\mathcal{I}^{\mathbb{P}}$-$LIM^rX_i$. Hence we deduce that $\mathcal{I}^{\mathbb{P}}$-$LIM^rX_i$ is a closed set.
\end{proof}

Before moving to our next result, let us denote $\bar{\theta}_r(X_*)=\{Y\in\mathfrak{X}:\mathbb{P}(d(X_*,Y)\geq r)=0)\}$ and $\bar{B}_r(X_*)=\{Y\in\mathfrak{X}:\rho(X_*,Y)\leq r)\}$, where $X_*\in \mathfrak{X}$ and $r\geq 0$. The forthcoming theorem rigorously establishes that the limit set $\mathcal{I}^{\mathbb{P}}$-$LIM^rX_i$ is confined between the sets $\bar{\theta}_r(X_*)$ and $\bar{B}_r(X_*)$, thereby elucidating its precise positional bounds within the underlying metric structure.
\begin{theorem}
	If $\{X_n\}_{n\in\mathbb{N}}$ is $\mathcal{I}$-convergent in probability to $X_*\in \mathfrak{X}$ then for each $r\geq0$,
	$$ \bar{\theta}_r(X_*)\subseteq\mbox{$\mathcal{I}^{\mathbb{P}}$-$LIM^rX_i$}\subseteq \bar{B}_r(X_*).$$
\end{theorem}
\begin{proof}
Assume that $\{X_n\}_{n\in\mathbb{N}}$ is $\mathcal{I}$-convergent in probability to $X_*\in \mathfrak{X}$. Let $\varepsilon,\delta>0$ be given. Then we have
$$
A(\varepsilon,\delta):=\{n\in\mathbb{N}:\mathbb{P}(d(X_n,X_*)>\varepsilon)>\delta\}\in\mathcal{I}.
$$
Now pick arbitrary $Y\in \bar{\theta}_r(X_*)$. So we get that $\mathbb{P}(d(X_*,Y)\geq r)=0$. Note that, we can write
\begin{align*}
\mathbb{P}(d(X_n,Y)>r+\varepsilon)&\leq \mathbb{P}(d(X_n,X_*)\geq\varepsilon)+\mathbb{P}(d(Y,X_*)\geq r)\\
&\leq  \mathbb{P}(d(X_n,X_*)\geq\varepsilon).
\end{align*}
With this, we have $\{n\in\mathbb{N}:\mathbb{P}(d(X_n,Y)>r+\varepsilon)>\delta\}\subseteq A(\varepsilon,\delta)$. Since 
  $A(\varepsilon,\delta)\in\mathcal{I}$, we conclude that $Y\in$$\mathcal{I}^{\mathbb{P}}$-$LIM^rX_i$, i.e., $\bar{\theta}_r(X_*)\subseteq\mathcal{I}^{\mathbb{P}}$-$LIM^rX_i$.\\
  
  Next, let us pick arbitrary $Y\in\mathcal{I}^{\mathbb{P}}$-$LIM^rX_i$. Therefore, we have 
  $$
  A(\varepsilon):=\{n\in\mathbb{N}:\mathbb{P}(d(X_n,Y)>r+\frac{\varepsilon}{2})>r+\frac{\varepsilon}{2}\}\in\mathcal{I}.
  $$
  Now, in view of Proposition \ref{pro 1}, we also have 
  $$
  B(\varepsilon):=\{n\in\mathbb{N}:\rho(X_n,X_*)>\frac{\varepsilon}{2}\}\in\mathcal{I}.
  $$
  Let us fix any $j\in \mathbb{N}\setminus(A(\varepsilon)\cup B(\varepsilon))$. Since $\mathbb{P}(d(X_j,Y)>r+\frac{\varepsilon}{2})\leq r+\frac{\varepsilon}{2}$, it is evident that $\rho(X_j,Y)\leq r+\frac{\varepsilon}{2}$. So, observe that 
  \begin{align*}
  \rho(X_*,Y)&\leq \rho(X_j,X_*)+\rho(X_j,Y)\\
  &\leq r+\varepsilon.
  \end{align*}
Since $\varepsilon>0$ was chosen arbitrarily, we have $\rho(X_*,Y)\leq r$, i.e., $Y\in \bar{B}_r(X_*)$. Thus we get that $\mathcal{I}^{\mathbb{P}}$-$LIM^rX_i\subseteq \bar{B}_r(X_*)$. Finally, by combining the preceding results, we may deduce that $ \bar{\theta}_r(X_*)\subseteq\mathcal{I}^{\mathbb{P}}$-$LIM^rX_i\subseteq \bar{B}_r(X_*).$
\end{proof}
We now present an example to illustrate that the set inclusions in the preceding result can be strict, as the next example substantiates our assertion. 
\begin{example}
	Let us consider the ideal $\mathcal{I}_\delta=\{A\subset\mathbb{N}:\delta(A)=0\}$. Suppose that $\{X_n\}_{n\in\mathbb{N}}$ is a sequence of i.i.d. random variables, where $X_n\sim Bernoulli(p)$ for each $n\in\mathbb{N}$ and $0<p<1$. Let us construct $\{Y_n\}_{n\in\mathbb{N}}$ such that
\begin{align*}
	Y_n=
	\begin{cases}
		2^n\displaystyle\prod_{i=1}^{n} X_i~&\mbox{ if }~ n\notin\{m^2:m\in\mathbb{N}\},\\
		\displaystyle\sum_{i=1}^{n}X_i ~&\mbox{ elsewhere}.
	\end{cases}
\end{align*}
Let $\varepsilon>0$ be given and  $n\in\mathbb{N}\setminus\{m^2:m\in\mathbb{N}\}$. Therefore, observe that 
\begin{align*}
	\mathbb{P}(|Y_n-0|>\varepsilon)&=\mathbb{P}(Y_n=2^n)\\
                               	   &=\mathbb{P}(X_i=1,\forall i\in [1,n])\\
                               	   &=\mathbb{P}(X_1=1)\mathbb{P}(X_2=1)...\mathbb{P}(X_n=1)\\
                               	   &=p^n\to 0~\mbox{ as }~n\to\infty.
\end{align*}
Consequently, $\{Y_n\}_{n\in\mathbb{N}}$ is $\mathcal{I}_\delta$-convergent in probability to the random variable $X_*=0$.\\

We now intend to show that $Y_n\xrightarrow[1]{\mathcal{I}_{\delta}^{\mathbb{P}}} Y_*$, where $\mathbb{P}(Y_*=1)=1$. Note that, for each $n\in\mathbb{N}\setminus\{m^2:m\in\mathbb{N}\}$, we have
\begin{align*}
\mathbb{P}(|Y_n-Y_*|<1+\varepsilon)&=\mathbb{P}(-\varepsilon<Y_n<2+\varepsilon)\\
                                   &=\mathbb{P}(Y_n=0)\\
                                   &=\mathbb{P}(X_i=0,~\mbox{ for some }~i\in [1,n])\\
                                   &=1-\mathbb{P}(X_i=1,\forall i\in [1,n])\\
                                   &=1-p^n\\
                                   \Rightarrow \mathbb{P}(|Y_n-Y_*|&\geq 1+\varepsilon)=p^n\to 0~\mbox{ as }~n\to\infty.
\end{align*}
Since $\{m^2:m\in\mathbb{N}\}\in\mathcal{I}$, we obtain that $Y_*\in\mathcal{I}^{\mathbb{P}}$-$LIM^{1}Y_i$. Now observe that $Y_*\notin \bar{\theta}_1(X_*)$ since $\mathbb{P}(|X_*-Y_*|\geq 1)=1$. This ensures that $\bar{\theta}_1(X_*)\subsetneq \mathcal{I}_{\delta}^{\mathbb{P}}$-$LIM^{1}Y_i$.\\

Next, we will show that $Z\notin \mathcal{I}_{\delta}^{\mathbb{P}}$-$LIM^{1}Y_i$, where $\mathbb{P}(Z=0)=\mathbb{P}(Z=2)$. Here, we choose $0<\varepsilon<\frac{1}{2}$. Then, for any $n\neq m^2$, we can write
\begin{align*}
\mathbb{P}(|Y_n-Z|>1+\varepsilon)=\mathbb{P}(\{(0,2),(2^n,0),(2^n,2)\})=\frac{3}{4}.
\end{align*}
This entails that $Z\notin \mathcal{I}_{\delta}^{\mathbb{P}}$-$LIM^{1}Y_i$. Now, it is evident that $Z\in \bar{B}_1(X_*)$ as $\bar{B}_1(X_*)=\mathfrak{X}(\mathbb{R})$. Consequently,  $\mathcal{I}_{\delta}^{\mathbb{P}}$-$LIM^{1}Y_i\subsetneq \bar{B}_1(X_*)$. Thus,  by combining the preceding results, we can conclude that $\bar{\theta}_1(X_*)\subsetneq\mathcal{I}_{\delta}^{\mathbb{P}}$-$LIM^{1}Y_i\subsetneq \bar{B}_1(X_*)$.
\end{example}
This section concludes with a demonstration that, for a certain class of ideals,  the set $\mathcal{I}^{\mathbb{P}}$-$LIM^rX_i$ is a Borel set of the  $F_{\sigma\delta}$ type in the metric space $(\mathfrak{X^0},\rho)$.
\begin{theorem}\label{theo1}
Let $r\geq 0$ and $\underline{X}=\{X_n\}_{n\in\mathbb{N}}$ be a sequence of random variables in  $\mathfrak{X}$. Then
for an analytic $P$-ideal $\mathcal{I}$, the limit set $\mathcal{I}^{\mathbb{P}}$-$LIM^r\underbar{X}$  is a Borel set of the  $F_{\sigma\delta}$ type in $(\mathfrak{X^0},\rho)$.
\end{theorem}
\begin{proof}
	Since $\mathcal{I}$ is an analytic $P$-ideal,  $\mathcal{I}=\{A\subseteq\mathbb{N}:\lim_{t\to\infty}\varphi(A\setminus\{1,2,...,t\})=0\}$ for some lscsm $\varphi$ on $\mathbb{N}$ (see  \cite{solecki}).
	For $k,t\in\mathbb{N}$, we consider the open set $$\Theta_{t,k}=\left\{X_*\in \mathfrak{X}:\rho (X_t,X_*)>r+\frac{1}{k}\right\}.$$
Since the infimum in the definition of $\rho(X_t, X_*)$ is attained, we can write
	$$
	\mathbb{P}(d(X_t,X_*)>\delta)>\delta\Leftrightarrow \rho(X_t,X_*)>\delta, ~\mbox{ for every }~\delta>0.
	$$
The definition of rough ideal convergence in probability ensures that
	\begin{align}\label{eq 1a}
		&\mathcal{I}^{\mathbb{P}}\mbox{-}LIM^r\underbar{X}\nonumber\\
			=& \left\{X_*\in\mathfrak{X^0}: (\forall k\in\mathbb{N})~ \{t\in\mathbb{N}:\mathbb{P}(d(X_t,X_*)>r+\frac{1}{k})>r+\frac{1}{k} \}\in\mathcal{I}\right\}\nonumber\\
		=& \left\{X_*\in\mathfrak{X^0}: (\forall k\in\mathbb{N})~ \{t\in\mathbb{N}:X_*\in \Theta_{t,k} \}\in\mathcal{I}\right\}\nonumber\\
		=&\bigcap_{k=1}^{\infty}\left\{X_*\in X:  \{t\in\mathbb{N}:X_*\in \Theta_{t,k} \}\in\mathcal{I}\right\}\nonumber\\
		=&\bigcap_{k=1}^{\infty}\left\{X_*\in\mathfrak{X^0}: \lim_{t\to\infty} \varphi(\{i\in\mathbb{N}:X_*\in \Theta_{i,k} \}\setminus\{1,2,...,t\})=0 \right\}\nonumber\\
		=&\bigcap_{k=1}^{\infty}\left\{X_*\in\mathfrak{X^0}: (\forall j\in\mathbb{N})~(\exists m\in\mathbb{N}): \varphi(\{i\in\mathbb{N}:X_*\in \Theta_{i,k} \}\setminus\{1,2,...,t\})\leq\frac{1}{j}~(\forall t\geq m) \right\}\nonumber\\
		=&\bigcap_{k=1}^{\infty}\bigcap_{j=1}^{\infty}\bigcup_{m=1}^{\infty}\bigcap_{t=m}^{\infty} \left\{X_*\in\mathfrak{X^0}:  \varphi(\{i\in\mathbb{N}:X_*\in \Theta_{i,k} \}\setminus\{1,2,...,t\})\leq\frac{1}{j}\right\}
	\end{align}
	For every fixed $t,j\in\mathbb{N},$ we consider the family  $$\mathscr{F}_{t,j}=\{F\subseteq \mathbb{N}\setminus\{1,2,..,t\}: \varphi{(F)}>\frac{1}{j}\}.$$
	Observe that for each \( F \in \mathscr{F}_{t,j} \), we have \( \varphi(F) > \frac{1}{j} \). Since \( \varphi \) is lower semicontinuous, \( \varphi(F) = \displaystyle\lim_{i \to \infty} \varphi(F \cap [1,i]) \). Then there exists \( i_0 \in \mathbb{N} \) such that \( \varphi(F \cap [1,i_0]) > \frac{1}{j} \). Therefore, without loss of any generality, we can assume that each \( F \in \mathscr{F}_{t,j} \) is finite.
	Then, notice that 
	\begin{align*}
		\mathfrak{A}_{k,j,t}=&\left\{X_*\in\mathfrak{X^0}:  \varphi(\{i\in\mathbb{N}:X_*\in \Theta_{i,k} \}\setminus\{1,2,...,t\})\leq\frac{1}{j}\right\}\\
		=&\left\{X_*\in\mathfrak{X^0}: (\forall F\in \mathscr{F}_{t,j})~(\exists i\in F)~\mbox{such that}~X_*\notin \Theta_{i,k}\right\}\\
		=& \bigcap_{F\in \mathscr{F}_{t,j}} \bigcup_{i\in F} \mathfrak{X^0}\setminus \Theta_{i,k}.
	\end{align*}
	Since $\Theta_{i,k}$ is an open set for every $i,k\in\mathbb{N}$ and $F$ is finite, it follows that  $\mathfrak{A}_{k,j,t}$  is  closed in  $(\mathfrak{X^0},\rho)$. Finally, in view of Eq (\ref{eq 1a}), we deduce that $\mathcal{I}^{\mathbb{P}}$-$LIM^r\underbar{X}$  is an $F_{\sigma\delta}$ subset of $(\mathfrak{X^0},\rho)$.
\end{proof}
%%%%%%%%%%%%%%%%%%%%%%%%%%%%%%%%%%%%%%%%%%%%%%%%%%%%%%%%%%%%%%%%%%%%%%%%%%%%%%%%%%%%%%%%%%%%%%%%%%%%%%%%%%%%%%%%%%%%%%%%
\section{Rough ideal limit points and rough ideal cluster points in probability}\label{secti3}

We begin this section by recalling two fundamental notions related to ideals on $\mathbb{N}$, namely 
$\mathcal{I}$-limit points and 
$\mathcal{I}$-cluster points, which are essential tools for our probabilistic extension of these ideas.
\begin{definition}\label{defcluster}\cite[Definition 4.1]{kostyrko2000convergence}
	Suppose $\{x_n\}_{n\in\mathbb{N}}$ is a sequence in a metric space $(U,d)$.
	\begin{itemize}
		\item[(i)] An element $y\in U$ is called an $\mathcal{I}$-limit point of $\{x_n\}_{n\in\mathbb{N}}$ if there exists $A\notin \mathcal{I}$ such that $\displaystyle\lim_{n\in A}x_n=y$.\\
		\item[(ii)] An element $y\in U$ is called an $\mathcal{I}$-cluster point of $\{x_n\}_{n\in\mathbb{N}}$, if  $$\{n\in\mathbb{N}:d(x_n,y)<\varepsilon\}\notin\mathcal{I}~\mbox{ for each }~ \varepsilon>0.$$
	\end{itemize}
\end{definition} 
\begin{remark}

Let us consider a probabilistic setting, as described in Remark \ref{impremark}, in which $\mathbb{P}(Y = y) = 1$ and $\mathbb{P}(X_n = x_n) = 1$ for each $n \in \mathbb{N}$. Then, observe that 
\begin{itemize}
\item[(a)] Within the framework of $\mathcal{I}$-limit points, it follows from Definition~\ref{defcluster}(i) that  
\begin{equation}\label{eqnlimitprobability}
	\lim_{n \in A} \mathbb{P}(d(X_n, Y) < \varepsilon) = 1 \quad \text{for some } A \notin \mathcal{I}.
\end{equation}
Note that Equation~(\ref{eqnlimitprobability}) implies that any random variable $Y \in \mathfrak{X}$ can be viewed as an $\mathcal{I}$-limit point in probability of the sequence $\{X_n\}_{n \in \mathbb{N}}$ in $\mathfrak{X}$, provided there exists a set $A \notin \mathcal{I}$ such that the subsequence $\{X_n\}_{n \in A}$ converges to $Y$ in probability.
\item[(b)] Subsequently, in the context of
$\mathcal{I}$-cluster points, we have from Definition \ref{defcluster}(ii)  that 
\begin{align}\label{eqnclusterprobability}
	&\{n \in \mathbb{N} : \mathbb{P}(d(X_n, Y) < \varepsilon) = 1\} \notin \mathcal{I}\nonumber\\
	\Rightarrow ~&\{n \in \mathbb{N} : \mathbb{P}(d(X_n, Y) < \varepsilon)> 1-\delta\} \notin \mathcal{I}~\mbox{ for every }~\delta>0. 
\end{align}
Observe that Eq (\ref{eqnclusterprobability}) suggests that any random variable $Y \in \mathfrak{X}$  can be regarded as an $\mathcal{I}$-cluster point in probability of the sequence $\{X_n\}_{n\in\mathbb{N}}$ in $\mathfrak{X}$, if for each $\varepsilon>0$, there exists a set $B \subseteq \mathbb{N}$ with $B \notin \mathcal{I}$ such that, for each $n \in B$, the probability of the event 
$$
\Omega_Y(\varepsilon):= \{\omega \in S : d(X_n(\omega), Y(\omega)) < \varepsilon\}
$$
is arbitrarily close to 1. 
\end{itemize}

These two observations naturally lead  us to define the notions of $\mathcal{I}$-limit points and $\mathcal{I}$-cluster points in probability for arbitrary sequences of random variables. Although our intuition based on convergence in probability might lead us to expect that the above two concepts are equivalent, this is not necessarily the case (see Example \ref{notclosed}).

Nevertheless, the formulation presented in Eq (\ref{eqnclusterprobability}) does not fully characterize all forms of $\mathcal{I}$-cluster in probability. In particular, it excludes cases when $a_Y <\mathbb{P}(\Omega_Y(\varepsilon))<b_Y$, with $0<a_Y,b_Y<1$ which depend solely on $Y$. For instance, consider the sequence $\{X_n\}_{n\in\mathbb{N}}$ in $\mathfrak{X}(\mathbb{R})$ such that 
$$\mathbb{P}(X_n=0)=\mathbb{P}(X_n=n)~\mbox{ for each }~ n\in\mathbb{N}.$$
Let us pick any $Y\in\mathfrak{X}(\mathbb{R})$ with $P(Y=0)=\frac{1}{2}=P(Y=1)$. Further, assume that for each $n\in\mathbb{N}$, the joint probabilities satisfy $$\mathbb{P}(X_n=0,Y=0)=\mathbb{P}(X_n=n,Y=0)=\mathbb{P}(X_n=0,Y=1)=\mathbb{P}(X_n=n,Y=1).$$
Then, for any admissible ideal $\mathcal{I}$, it follows that 
$$
\mathbb{N}\setminus\{\mbox{finite set\}}=\{n \in \mathbb{N} : \mathbb{P}(|X_n-Y| < \varepsilon) = \frac{1}{4}\} \notin \mathcal{I}~\mbox{ for every }~0<\varepsilon<1.
$$

\end{remark}
This limitation motivates the introduction of two refined notions of $\mathcal{I}$-cluster points in probability. Now we are in a position to define the notions of rough $\mathcal{I}$-limit points and rough $\mathcal{I}$-cluster points in probability for sequences of random variables. 
	
\begin{definition}\label{cluster-limit}
Let $r\geq 0$, $\mathcal{I}$ be an ideal on $\mathbb{N}$, and $\underline{X}=\{X_n\}_{n\in\mathbb{N}}$ be a sequence of random variables in  $\mathfrak{X}$.
\begin{itemize}
\item[(a)] A random variable $Y\in\mathfrak{X}$ is called a rough $\mathcal{I}$-limit point  of $\underline{X}$ in probability with roughness degree $r$ (briefly, $r-\mathcal{I}^{\mathbb{P}}$ limit point) if there exists  $A\notin \mathcal{I}$ such that 
$$
\displaystyle\lim_{n\in A}\mathbb{P}(d(X_n,Y)\geq r+\varepsilon)=0~\mbox{ for every }~\varepsilon>0.
$$
We denote $\Lambda^{r}_{\underline{X}}(\mathcal{I}^{\mathbb{P}})$ as the set of  $r-\mathcal{I}^{\mathbb{P}}$ limit points  of $\underline{X}$ in probability.

\item[(b1)]  A random variable $Y\in\mathfrak{X}$ is called a strong rough $\mathcal{I}$-cluster point of $\underline{X}$ in probability with roughness degree $r$ (briefly, $r^s-\mathcal{I}^{\mathbb{P}}$ cluster point), provided that
$$\{n\in\mathbb{N}:\mathbb{P}(d(X_n,Y)<r+\varepsilon)>1-\delta\}\notin\mathcal{I}~\mbox{ for every }~\varepsilon,\delta>0.
$$ 
We will denote $\Gamma^{r^s}_{\underline{X}}(\mathcal{I}^{\mathbb{P}})$ as the set of $r^s-\mathcal{I}^{\mathbb{P}}$ cluster points  of $\underline{X}$ in probability.
\item[(b2)] A random variable $Y\in\mathfrak{X}$ is called a weak rough $\mathcal{I}$-cluster point of $\underline{X}$ in probability with roughness degree $r$ (briefly, $r^w-\mathcal{I}^{\mathbb{P}}$ cluster point), if there exists a fixed $\delta_{*}=\delta_{*}(Y)>0$ such that $$
\{n\in\mathbb{N}:\mathbb{P}(d(X_n,Y)<r+\varepsilon)>\delta_*\}\notin\mathcal{I}~\mbox{ for every }~\varepsilon>0.
$$ The positive real number $\delta_{*}(Y)$ represents $r^w-\mathcal{I}^{\mathbb{P}}$ cluster point constant connected to $Y$ w.r.t. the sequence $\underline{X}$. We will denote $\Gamma^{r^w}_{\underline{X}}(\mathcal{I}^{\mathbb{P}})$ as the set of $r^w-\mathcal{I}^{\mathbb{P}}$ cluster points  of $\underline{X}$ in probability. 
\end{itemize}
\end{definition}	
Note that for each admissible ideal $\mathcal{I}$, the following inclusions hold:
$$
\Lambda^{r}_{\underline{X}}(\mathcal{I}^{\mathbb{P}})\subseteq \Gamma^{r^s}_{\underline{X}}(\mathcal{I}^{\mathbb{P}})\subseteq \Gamma^{r^w}_{\underline{X}}(\mathcal{I}^{\mathbb{P}}).
$$
Indeed, these inclusions can be proper, as demonstrated by the subsequent examples which substantiate this claim. 
\begin{example}\label{notclosed}
	Consider the admissible ideal $\mathcal{I}=\mathcal{I}_\delta$.  For each $j\in\mathbb{N}$, we set $A_j=\{2^{j-1}(2k+1):k\in\mathbb{N}\}$. Note that $\{A_j\}_{j\in\mathbb{N}}$ forms a partition of $\mathbb{N}$ and $A_j\notin\mathcal{I}$ for any $j\in\mathbb{N}$ since $\delta(A_j)=\frac{1}{2^j}>0$. Now let us define $\underline{X}=\{X_n\}_{n\in\mathbb{N}}$ in $\mathfrak{X}(\mathbb{R})$ as follows:
	\begin{align*}
		P(X_{n}=\frac{1}{j})=1-\frac{1}{n^2}~\mbox{ and }~ P(X_{n}=\frac{1}{j+1})=\frac{1}{n^2}~\mbox{ if }~n\in A_j.
	\end{align*}
Let $\{Y_n\}_{n\in\mathbb{N}}\in\mathfrak{X}(\mathbb{R})$ such that $\mathbb{P}(Y_j=\frac{1}{j})=1$. 
Observe that, for each $\varepsilon>0$, 
$$
\mathbb{P}(|X_n-Y_j|>\varepsilon)\leq\mathbb{P}(X_n\neq \frac{1}{j})=\frac{1}{n^2}~\mbox{ if }~n\in A_j.
$$
This ensures that $Y_j\in \Lambda^{0}_{\underline{X}}(\mathcal{I}^{\mathbb{P}})$ for each $j\in\mathbb{N}$. Note that $Y_j\in \Gamma^{0^s}_{\underline{X}}(\mathcal{I}^{\mathbb{P}})$ for each $j\in\mathbb{N}$ and $Y_j\xrightarrow{\rho}Z$ where $\mathbb{P}(Z=0)=1$. Since $\Gamma^{0^s}_{\underline{X}}(\mathcal{I}^{\mathbb{P}})$ is closed in $(\mathfrak{X^0(\mathbb{R})},\rho)$ (see Proposition \ref{strong closed}), we have $Z\in \Gamma^{0^s}_{\underline{X}}(\mathcal{I}^{\mathbb{P}})$. 

We now aim to show that $Z\notin \Lambda^{0}_{\underline{X}}(\mathcal{I}^{\mathbb{P}})$. Let $j\in\mathbb{N}$ be arbitrary. Assume, on the contrary, that there exists $A\notin\mathcal{I}$ such that $\displaystyle\lim_{n\in A}\mathbb{P}(|X_n-Z|\geq\frac{1}{j})=0$. Observe that
\begin{align*}
A&= \{k\in A:\mathbb{P}(|X_k-Z|\geq\frac{1}{j})=1\}\cup \{k\in A:\mathbb{P}(|X_k-Z|<\frac{1}{j})=1\}\\
&\subseteq \{k\in A:\mathbb{P}(|X_k-Z|\geq\frac{1}{j})=1\}\cup \{k\in \mathbb{N}:\mathbb{P}(|X_k-Z|<\frac{1}{j})=1\}\\
	&=\{\mbox{finite set}\}\cup\bigcup_{n=j+1}^{\infty}A_n.
\end{align*}	
Subsequently, we obtain
$$
0\leq\bar{\delta}(A)\leq \displaystyle\sum_{n=j+1}^{\infty}\bar{\delta}(A_n)\leq \frac{1}{2^j}<\frac{1}{j}.
$$

Since $j\in\mathbb{N}$ was chosen arbitrarily, we have $\delta(A)=0$, i.e., $A\in\mathcal{I}$ - which is a contradiction. Thus we can conclude that $Z\in\Gamma^{0^s}_{\underline{X}}(\mathcal{I}^{\mathbb{P}})\setminus\Lambda^{0}_{\underline{X}}(\mathcal{I}^{\mathbb{P}})$.
\end{example}
\begin{note}
 Example \ref{notclosed} also ensures that, for a sequence $\underline{X}$ in $\mathfrak{X}$, in general, the set $\Lambda^{r}_{\underline{X}}(\mathcal{I}^{\mathbb{P}})$ is not closed in  $(\mathfrak{X^0},\rho)$. 
\end{note}
It is a straightforward consequence that, when each element of the sequence $\underline{X}$ 
has a degenerate (one-point) distribution, one always has the equality $\Gamma^{r^s}_{\underline{X}}(\mathcal{I}^{\mathbb{P}})= \Gamma^{r^w}_{\underline{X}}(\mathcal{I}^{\mathbb{P}})$. We now present an example in which this equality does not hold.
\begin{example}
First, consider an admissible ideal $\mathcal{I}$ on $\mathbb{N}$ and then fix an $A\subset\mathbb{N}$ such that $A\notin\mathcal{I}$. Next, we define a sequence $\underline{X}=\{X_n\}_{n\in\mathbb{N}}$ in $\mathfrak{X}(\mathbb{R})$ as follows:
\begin{align*}
	X_{n}\in
	\begin{cases}
		\{-2,1\}~\mbox{ with }~ \mathbb{P}(X_n=-2)=\frac{n^2-1}{2n^2}~\mbox{ and }~\mathbb{P}(X_n=1)=\frac{n^2+1}{2n^2}~&\mbox{ if }~n\in A, \\
		\{-1,n^2\}~\mbox{ with }~ \mathbb{P}(X_n=-1)=\frac{1}{n}~\mbox{ and }~ \mathbb{P}(X_n=n^2)=1-\frac{1}{n}~&\mbox{ if }~n\notin A.
	\end{cases}
\end{align*}
Let $0<\varepsilon<1$ be arbitrary, $r=1$  and $Y\sim Bernoulli(\frac{1}{2})$, i.e., $\mathbb{P}(Y=0)=\mathbb{P}(Y=1)$. Then,  observe that
\begin{align*}
	\lim_{n\in A}\mathbb{P}(|X_n-Y|<r+\varepsilon)=\lim_{n\in A}\mathbb{P}(Y-r-\varepsilon<X_n<Y+r+\varepsilon)=\frac{1}{2}.
\end{align*}
Thus, we have 
$$
A\subset^*\{n\in\mathbb{N}:\mathbb{P}(|X_n-Y|<r+\varepsilon)=\frac{1}{2}\}.
$$
As a consequence, we get that $\{n\in\mathbb{N}:\mathbb{P}(|X_n-Y|<r+\varepsilon)=\frac{1}{2}\}\notin \mathcal{I}$. Note also that 
$\displaystyle\lim_{n\in\mathbb{N}\setminus A}\mathbb{P}(|X_n-Y|<r+\varepsilon)=0$.
Since $\mathcal{I}$ is admissible, we obtain that 
$$
\{n\in\mathbb{N}:\mathbb{P}(|X_n-Y|<r+\varepsilon)>1-\delta\}\in\mathcal{I}~\mbox{ whenever }~0<\delta<\frac{1}{3}.
$$
 This ensures that $Y\in \Gamma^{r^w}_{\underline{X}}(\mathcal{I}^{\mathbb{P}})\setminus \Gamma^{r^s}_{\underline{X}}(\mathcal{I}^{\mathbb{P}})$. 
\end{example}
\begin{definition}
Two sequences of random variables $\underline{X}$ and $\underline{Y}$ are said to be $\mathcal{I}$-almost surely (briefly, $\mathcal{I}$ a.s.) if $\{n\in\mathbb{N}:\mathbb{P}(X_n=Y_n)=1\}\in \mathcal{F}(\mathcal{I})$.
\end{definition}
The following result guarantees that if $\underline{X}$ and $\underline{Y}$ are $\mathcal{I}$ a.s. then their respective rough $\mathcal{I}$-limit point sets and rough $\mathcal{I}$-cluster point sets coincide.
%%%%%%%%%%%%%%%%%%%%%%%%%%%%%%%%%%%%%%%%%%%%%%%%%%%%%%%%%%%%%%%%%%%%%%%%%%%%%%%%%%%%%%%%%%%%%
\footnotetext{For each $A\subseteq\mathbb{N}$, $\bar{\delta}(A)=\displaystyle\limsup_{n\to\infty}\frac{|A\cap\{1,2,...,n\}|}{n}$ is the upper natural density of $A$.}
%%%%%%%%%%%%%%%%%%%%%%%%%%%%%%%%%%%%%%%%%%%%%%%%%%%%%%%%%%%%%%%%%%%%%%%%%%%%%%%%%%%%%%%%%%%%%%%%%%%%%%
\begin{proposition}
	If $\underline{X}$ and $\underline{Y}$ are $\mathcal{I}$ a.s. sequences in $\mathfrak{X}$ then $\Lambda^{r}_{\underline{X}}(\mathcal{I}^{\mathbb{P}})=\Lambda^{r}_{\underline{Y}}(\mathcal{I}^{\mathbb{P}})$,  $\Gamma^{r^s}_{\underline{X}}(\mathcal{I}^{\mathbb{P}})=\Gamma^{r^s}_{\underline{Y}}(\mathcal{I}^{\mathbb{P}})$,  and $\Gamma^{r^w}_{\underline{X}}(\mathcal{I}^{\mathbb{P}})=\Gamma^{r^w}_{\underline{Y}}(\mathcal{I}^{\mathbb{P}})$ for every $r\geq 0$.
\end{proposition}

\begin{proof}
Assume that $\{n\in\mathbb{N}:\mathbb{P}(X_n=Y_n)<1\}\in \mathcal{I}$ and pick a $\xi\in\Lambda^r_{\underline{X}}(\mathcal{I}^{\mathbb{P}})$. Let $\varepsilon>0$ be arbitrary. Then there exists $A\notin \mathcal{I}$ such that
$$
\displaystyle\lim_{n\in A}\mathbb{P}(d(X_n,\xi)\geq r+\frac{\varepsilon}{2})=0.
$$  
Observe that $$\{n\in A:\mathbb{P}(X_n=Y_n)<1\}\subseteq \{n\in\mathbb{N}:\mathbb{P}(X_n=Y_n)<1\}.$$
Consequently, $A'=\{n\in A:\mathbb{P}(X_n=Y_n)=1\}\notin \mathcal{I}$ as $A\notin \mathcal{I}$.
 We intend to show that $\displaystyle\lim_{n\in A'}\mathbb{P}(d(Y_n,\xi)\geq r+\frac{\varepsilon}{2})=0$. This follows from the inequality that for each $n\in A'$, 
\begin{align*}
\mathbb{P}(d(Y_n,\xi)\geq r+\varepsilon)&\leq \mathbb{P}(d(X_n,\xi)\geq r+\frac{\varepsilon}{2})+\mathbb{P}(d(Y_n,X_n)\geq
\frac{\varepsilon}{2})\\
&=\mathbb{P}(d(X_n,\xi)\geq r+\frac{\varepsilon}{2})+\mathbb{P}(X_n\neq Y_n)=\mathbb{P}(d(X_n,\xi)\geq r+\frac{\varepsilon}{2}).
\end{align*}	
This ensures that $\Lambda^r_{\underline{X}}(\mathcal{I}^{\mathbb{P}})\subseteq\Lambda^r_{\underline{Y}}(\mathcal{I}^{\mathbb{P}})$. Conversely, by interchanging $\underline{X}$ and $\underline{Y}$, we derive $\Lambda^r_{\underline{Y}}(\mathcal{I}^{\mathbb{P}})\subseteq\Lambda^r_{\underline{X}}(\mathcal{I}^{\mathbb{P}})$. As a consequence, we obtain $\Lambda^r_{\underline{X}}(\mathcal{I}^{\mathbb{P}})=\Lambda^r_{\underline{Y}}(\mathcal{I}^{\mathbb{P}})$.\\

Next, suppose that $\zeta\in\Gamma^{r^s}_{\underline{X}}(\mathcal{I}^{\mathbb{P}})$. Then for every $\varepsilon,\delta>0$, we have  
$$
B=\{n\in\mathbb{N}:\mathbb{P}(d(X_n,\zeta)<r+\frac{\varepsilon}{2})>1-\delta\}\notin \mathcal{I}.
$$ Note that  $B'=\{n\in B:\mathbb{P}(X_n=Y_n)=1\}\notin \mathcal{I}$. Then, for all $n\in B'$, we will obtain  
$$
\mathbb{P}(d(Y_n,\zeta)>r+\varepsilon)\leq \mathbb{P}(d(X_n,\zeta)>r+\frac{\varepsilon}{2}).
$$ 
Subsequently,
$$
\{n\in B':\mathbb{P}(d(X_n,\zeta)<r+\frac{\varepsilon}{2})>1-\delta\}\subseteq \{n\in B':\mathbb{P}(d(Y_n,\zeta)<r+\varepsilon)>1-\delta\}.
$$
Observe that 
$\{n\in B':\mathbb{P}(d(X_n,\zeta)<r+\frac{\varepsilon}{2})>1-\delta\}\notin \mathcal{I}$
since $B\setminus B'\in\mathcal{I}$. Therefore, 
we obtain  $\{n\in\mathbb{N}:\mathbb{P}(d(Y_n,\zeta)<r+\varepsilon)>1-\delta\}\notin \mathcal{I}$, i.e., $\zeta\in\Gamma_{\underline{Y}}(\mathcal{I}^{\mathbb{P}})$. As a consequence, $\Gamma_{\underline{X}}(\mathcal{I}^{\mathbb{P}})\subseteq\Gamma_{\underline{Y}}(\mathcal{I}^{\mathbb{P}})$. Finally, by symmetry, we  conclude that $\Gamma^{r^s}_{\underline{X}}(\mathcal{I}^{\mathbb{P}})=\Gamma^{r^s}_{\underline{Y}}(\mathcal{I}^{\mathbb{P}})$.

The claim that $\Gamma^{r^w}_{\underline{X}}(\mathcal{I}^{\mathbb{P}})=\Gamma^{r^w}_{\underline{Y}}(\mathcal{I}^{\mathbb{P}})$ can be demonstrated using a similar argument.
\end{proof}
%%%%%%%%%%%%%%%%%%%%%%%%%%%%%%%%%%%%%%%%%%%%%%%%%%%%%%%%%%%%%%%%%%%%%%%%%%%%%%%%%%%%%%%%%%%%%%%%%%%%%%%%%%%%%%%%%%%%%%%%%%%%%%%%%%%%%%%%%%%%
The subsequent result, under specific conditions, ensures the non-voidness of the set $\Gamma^{r^s}_{\underline{X}}(\mathcal{I}^{\mathbb{P}})$ for all positive values of $r$. Before we proceed, let us present an important lemma.
\begin{lemma}\label{lemma 2.10}
	For $Y\in \mathfrak{X}$, let $\varepsilon_Y,\delta_Y>0$ be such that
$\{n\in\mathbb{N}:\mathbb{P}(d(X_n,Y)<\varepsilon_Y)>\delta_{Y}\}\in\mathcal{I}$, where $\underline{X}$ is a sequence in $\mathfrak{X}$. Then $\{n\in\mathbb{N}:\mathbb{P}(d(X_n,Y)<\varepsilon_*)>\delta_{*}\}\in\mathcal{I}$ for every $0<\varepsilon_*\leq\varepsilon_Y$ and $\delta_Y\leq\delta_*$.
\end{lemma}
\begin{proof}
The proof follows from the fact that $A\in\mathcal{I}$ for all $A\subseteq B$ whenever $B\in\mathcal{I}$.
\end{proof}

\begin{proposition}\label{pro2.11}
	Let $\mathcal{I}$ be an ideal on $\mathbb{N}$ and $\underline{X}$ be a sequence of elements from $\mathfrak{X}$.	If $\mathfrak{B}$ is a compact set in $(\mathfrak{X}^0, \rho)$ such that $\{n \in \mathbb{N} : X_n \in \mathfrak{B}\} \notin \mathcal{I}$, then $\Gamma^{r^s}_{\underline{X}}(\mathcal{I}^{\mathbb{P}}) \neq \varnothing$ for any $r >0$.
\end{proposition}
\begin{proof}
	Assume, on the contrary, that $\Gamma^{r_*^s}_{\underline{X}}(\mathcal{I}^{\mathbb{P}})=\varnothing$ for some $r_*>0$. Then, for each $Y\in\mathfrak{B}$, there exist $\varepsilon_Y>0$ and $\delta_Y>0$ such that 
	\begin{equation}\label{eq*}
		\{n\in\mathbb{N}:\mathbb{P}(d(X_n,Y)<r_*+\varepsilon_Y)>1-\delta_{Y}\}\in\mathcal{I}.
	\end{equation}
We set $\varepsilon_*=\displaystyle\inf_{Y\in\mathfrak{B}}\varepsilon_Y$. Note that, without loss of generality, we can assume $0<\delta_Y<r_*$ for each $Y\in\mathfrak{B}$. It is evident that $\{B_{\rho}(Y,\delta_Y):Y\in\mathfrak{B}\}$ is an open cover for $\mathfrak{B}$, where $$B_{\rho}(X,r)=\{Y\in\mathfrak{X}^0:\rho(X,Y)<r\}.$$   
	%Then, by Lemma \ref{lemma 2.10}, we obtain 
%	\begin{equation}\label{eq 2}
		%\{n\in\mathbb{N}:\mathbb{P}(d(X_n,Y)<\mu)>1-\mu\}\in\mathcal{I}.
%	\end{equation}
	Since $\mathfrak{B}$ is compact in $(\mathfrak{X}^0, \rho)$, there exists $m\in\mathbb{N}$ such that
	\begin{align*}
		\mathfrak{B}\subseteq\bigcup_{i=1}^{m}B_{\rho}(Y_i,\delta_{Y_i}).
	\end{align*}
	With this, we have
	\begin{equation}\label{eq3}
	\{n\in\mathbb{N}:X_n\in\mathfrak{B}\}\subseteq\bigcup_{i=1}^{m}\{n\in\mathbb{N}:X_n\in B_{\rho}(Y_i,\delta_{Y_i})\}.
	\end{equation}
	Observe that
	\begin{align}\label{eq4}
		\mbox{if}~	&X_n\in  B_{\rho}(Y_i,\delta_{Y_i})~\mbox{ for some }~i\in\mathbb{N}\nonumber\\
		\Rightarrow& ~\mathbb{P}(d(X_n,Y_i)>\alpha)\leq\alpha~\mbox{ and }~\alpha<\delta_{Y_i}, ~\mbox{ where }~\alpha=\rho(X_n,Y_i)\nonumber\\
		\Rightarrow&~\mathbb{P}(d(X_n,Y_i)\leq\alpha)>1-\alpha~\mbox{ and }~\alpha<\delta_{Y_i}\nonumber\\
		\Rightarrow&~\mathbb{P}(d(X_n,Y_i)<\delta_{Y_i})>1-\delta_{Y_i}.
	\end{align}
 Since $\delta_{Y_i}<r_*+\varepsilon_{Y_i}$ for each $i\in\mathbb{N}$, in view of Lemma \ref{lemma 2.10}, Eq (\ref{eq3}), and Eq (\ref{eq4}), we obtain
		$$
		\{n\in\mathbb{N}:X_n\in\mathfrak{B}\}\subseteq\bigcup_{i=1}^{m}\{n\in\mathbb{N}:\mathbb{P}(d(X_n,Y_i)<\delta_{Y_i})>1-\delta_{Y_i}\}\in\mathcal{I},
		$$
		$-$ which is a contradiction.
Thus we can conclude that $\Gamma^{r^s}_{\underline{X}}(\mathcal{I}^{\mathbb{P}})\neq\varnothing$ for any $r>0$.
\end{proof}
Observe that, for any $\underline{X}\in\mathfrak{X}$, we already have $\Gamma^{r^s}_{\underline{X}}(\mathcal{I}^{\mathbb{P}})\subseteq \Gamma^{r^w}_{\underline{X}}(\mathcal{I}^{\mathbb{P}})$. Consequently, $\Gamma^{r^w}_{\underline{X}}(\mathcal{I}^{\mathbb{P}})\neq\varnothing$ for any $r>0$ whenever the conditions of Proposition \ref{pro2.11} are satisfied. So
a natural question that arises is whether these conditions can be weakened while still ensuring the non-emptiness of $\Gamma^{r^w}_{\underline{X}}(\mathcal{I}^{\mathbb{P}})$. In the next result, we establish that the assumption of compactness used in Proposition \ref{pro2.11} can, in fact, be replaced by the weaker condition of total boundedness.
\begin{proposition}Let $\mathcal{I}$ be an ideal on $\mathbb{N}$ and $\underline{X}$ be a sequence with values in $\mathfrak{X}$.	If $\mathfrak{B}$ is a totally bounded set in $(\mathfrak{X}^0, \rho)$ such that $\{n \in \mathbb{N} : X_n \in \mathfrak{B}\} \notin \mathcal{I}$, then $\Gamma^{r^w}_{\underline{X}}(\mathcal{I}^{\mathbb{P}}) \neq \varnothing$ for any $r>0$.
\end{proposition}
\begin{proof}
	Assume, on the contrary, that $\Gamma^{r_*^w}_{\underline{X}}(\mathcal{I}^{\mathbb{P}})=\varnothing$ for some $r_*>0$. Then, for each $Y\in\mathfrak{B}$, there exists $\varepsilon_Y>0$ such that 
	\begin{equation*}
		\{n\in\mathbb{N}:\mathbb{P}(d(X_n,Y)<r_*+\varepsilon_Y)>\delta\}\in\mathcal{I}~\mbox{ for every }~\delta>0.
	\end{equation*}
	We set $\varepsilon_*=\displaystyle\inf_{Y\in\mathfrak{B}}\varepsilon_Y$.  Then, in view of Lemma \ref{lemma 2.10}, we obtain 
	
	\begin{align}\label{eq6}
		&\{n\in\mathbb{N}:\mathbb{P}(d(X_n,Y)<r_*+\varepsilon_*)>\delta\}\in\mathcal{I}~\mbox{ for every }~\delta>0.
	\end{align}
	Since $\mathfrak{B}$ is  totally bounded in $(\mathfrak{X}^0, \rho)$, there exist $Y_1,Y_2,...,Y_m\in \mathfrak{X}^0$ such that
	\begin{align*}\label{eq 3}
		&\mathfrak{B}\subseteq\bigcup_{i=1}^{m}B_{\rho}(Y_i,r_*+\varepsilon_*)\nonumber\\
		\Rightarrow~&	\{n\in\mathbb{N}:X_n\in\mathfrak{B}\}\subseteq\bigcup_{i=1}^{m}\{n\in\mathbb{N}:X_n\in B_{\rho}(Y_i,r_*+\varepsilon_*)\}\\
		\Rightarrow~&\{n\in\mathbb{N}:X_n\in\mathfrak{B}\}\subseteq\bigcup_{i=1}^{m}\{n\in\mathbb{N}:\mathbb{P}(d(X_n,Y_i)< r_*+\varepsilon_*)>1-(r_*+\varepsilon_*)\}\in\mathcal{I}~\mbox{(by Eq (\ref{eq6}))}
	\end{align*}
	$-$ which is a contradiction.
Thus we can infer that $\Gamma^{r^w}_{\underline{X}}(\mathcal{I}^{\mathbb{P}})\neq\varnothing$ for any $r>0$.
\end{proof}
%%%%%%%%%%%%%%%%%%%%%%%%%%%%%%%%%%%%%%%%%%%%%%%%%%%%%%%%%%%%%%%%%%%%%%%%%%%%%%%%%%%%%%%%%%%%%%%%%%%%%%%%%%%%%%%%%%%%%%%%%%%%%%%%%%%%%%%

The next two results concern the closedness properties of the sets $\Gamma^{r^s}_{\underline{X}}(\mathcal{I}^{\mathbb{P}})$ and $\Gamma^{r^w}_{\underline{X}}(\mathcal{I}^{\mathbb{P}})$ in the metric space $(\mathfrak{X^0},\rho)$. It turns out that $\Gamma^{r^s}_{\underline{X}}(\mathcal{I}^{\mathbb{P}})$ is always closed,  whereas $\Gamma^{r^w}_{\underline{X}}(\mathcal{I}^{\mathbb{P}})$ is only conditionally closed.

\begin{proposition}\label{strong closed}
	Suppose $r\geq 0$ and $\mathcal{I}$ is an admissible ideal on $\mathbb{N}$. Then for any $\underline{X}$ in $\mathfrak{X}$, the set  $\Gamma^{r^s}_{\underline{X}}(\mathcal{I}^{\mathbb{P}})$ is closed in $(\mathfrak{X^0},\rho)$. 
\end{proposition}
\begin{proof}
	
Let $\varepsilon,\delta>0$ be arbitrary.
Assume that $\{Y_n\}_{n\in\mathbb{N}}$ is a sequence in $\Gamma^{r^s}_{\underline{X}}(\mathcal{I}^{\mathbb{P}})$ such that $Y_n\xrightarrow{\rho} Z$.  Since $\rho$ metrizes convergence in probability, there exists a natural $n_0$ such that
$$
\mathbb{P}(d(Y_{n_0},Z)\geq\frac{\varepsilon}{2})<\frac{\delta}{2}.
$$
Since $Y_{n_0}\in\Gamma^{r^s}_{\underline{X}}(\mathcal{I}^{\mathbb{P}})$, we obtain
$$
A(\varepsilon)=\{n\in\mathbb{N}:\mathbb{P}(d(X_n,Y_{n_0})<r+\frac{\varepsilon}{2})>1-\frac{\delta}{2}\}\notin\mathcal{I}
$$
Now pick any $n\in A(\varepsilon)$, then
\begin{align*}
	\mathbb{P}(d(X_n,Z)\geq r+\varepsilon)&\leq  \mathbb{P}(d(X_n,Y_{n_0})\geq r+\frac{\varepsilon}{2})+	\mathbb{P}(d(Y_{n_0},Z)\geq\frac{\varepsilon}{2})\\
	&\leq \mathbb{P}(d(X_n,Y_{n_0})\geq r+\frac{\varepsilon}{2})+\frac{\delta}{2}\\
	\Rightarrow 	\mathbb{P}(d(X_n,Z)<r+\varepsilon)&>\mathbb{P}(d(X_n,Y_{n_0})<r+\frac{\varepsilon}{2})-\frac{\delta}{2}>1-\delta.
\end{align*}
This ensures that  $\{n\in\mathbb{N}:\mathbb{P}(d(X_n,Z)<r+\varepsilon)>1-\delta\}\supseteq A(\varepsilon)$. Since $A(\varepsilon)\notin\mathcal{I}$, we have $\{n\in\mathbb{N}:\mathbb{P}(d(X_n,Z)<r+\varepsilon)>1-\delta\}\notin\mathcal{I}$, i.e., $Z\in\Gamma^{r^s}_{\underline{X}}(\mathcal{I}^{\mathbb{P}})$. Thus we can conclude that $\Gamma^{r^s}_{\underline{X}}(\mathcal{I}^{\mathbb{P}})$ is closed in $(\mathfrak{X^0},\rho)$. 
\end{proof}

In general, for a given sequence $\underline{X}$ in $\mathfrak{X}$, the set $\Gamma^{r^w}_{\underline{X}}(\mathcal{I}^{\mathbb{P}})$ is not closed in $(\mathfrak{X^0},\rho)$. We now present an example to demonstrate this point.
\begin{example}\label{not closed}
	Consider an admissible ideal $\mathcal{I}$ on $\mathbb{N}$ and fix any $A\notin\mathcal{I}$. Let $\{a_n\}_{n\in\mathbb{N}}$ be a sequence  of positive real number such that $a_n>a_{n+1}$ for each $\mathbb{N}$ and $\displaystyle\sum_{n=1}^{\infty}a_n=1$.  Next, consider the sequence $\underline{X}=\{X_n\}_{n\in\mathbb{N}}$, set up as follows:
	\begin{align*}
		X_{n}\in
		\begin{cases}
			\{a_1,a_2,...,a_k,...\}~\mbox{ with }~ \mathbb{P}(X_n=a_k)=a_k~\mbox{ for all }~k\in\mathbb{N}~&\mbox{ if }	~n\in A\\
			\{-n^3,n^2\}~\mbox{ with }~ \mathbb{P}(X_n=-n^3)=\frac{1}{3}~\mbox{ and }~\mathbb{P}(X_n=n^2)=\frac{2}{3}~&\mbox{ if }~n\notin A.\\
		\end{cases}
	\end{align*}
	Let $\{Y_n\}_{n\in\mathbb{N}}$ be a sequence in $\mathfrak{X}(\mathbb{R})$ such that 
	$$
	\mathbb{P}(Y_k=a_k)=1~\mbox{ for each }~k\in\mathbb{N}.
	$$
	We intend to show that $Y_k\in \Gamma^{0^w}_{\underline{X}}(\mathcal{I}^{\mathbb{P}})$ for each $k\in\mathbb{N}$. Let $\varepsilon>0$ be arbitrary. Then, for each $n\in A$, we have
	\begin{align*}
		\mathbb{P}(|X_n-Y_k|<\varepsilon)&=\mathbb{P}(Y_k-\varepsilon<X_n<Y_k+\varepsilon)\\
		&\geq \mathbb{P}(X_n=a_k)=a_k.
	\end{align*}
	This ensures that $A\subseteq \{n\in\mathbb{N}:\mathbb{P}(|X_n-Y_k|<\varepsilon)\geq \delta_{*}(Y_k)\}$, where $\delta_{*}(Y_k)=a_k$. 
	% As a consequence, we obtain  $Y_k\in\Gamma_{\underline{X}}(\mathcal{I}^{\mathbb{P}})$ as we already have $A\notin\mathcal{I}$. Since $\sum_{n=1}^{\infty}a_n=1$, we must have $\lim_{n\to\infty}a_n=0$. 
	Now, it is easy to realize that $Y_n\xrightarrow{\rho} Z$, where $\mathbb{P}(Z=0)=1$.\\
	
To prove $\Gamma^{0^w}_{\underline{X}}(\mathcal{I}^{\mathbb{P}})$ is not closed, it is enough to show that $Z\notin \Gamma^{0^w}_{\underline{X}}(\mathcal{I}^{\mathbb{P}})$.
	Since $\displaystyle\lim_{n\to\infty}a_n=0$, there exists $n_{\varepsilon}\in\mathbb{N}$ such that for all $n\geq n_{\varepsilon}$, we have  $0<a_n<\varepsilon$.
	Therefore, 
	\begin{align*}
		\mathbb{P}(|X_n-Z|<\varepsilon)&=\sum_{n=n_{\varepsilon}}^{\infty}a_n\to 0~\mbox{ as }~\varepsilon\to 0.
	\end{align*}
 Hence we can conclude that $\Gamma^{0^w}_{\underline{X}}(\mathcal{I}^{\mathbb{P}})$ is not closed in $(\mathfrak{X^0}(\mathbb{R}),\rho)$.
\end{example}
Observe that, in Example \ref {not closed}, we have $\inf\{\delta_{*}(Y):Y\in\Gamma^{0^w}_{\underline{X}}(\mathcal{I}^{\mathbb{P}})\}=0$. We now impose a restriction on $\Gamma^{r^w}_{\underline{X}}(\mathcal{I}^{\mathbb{P}})$ to ensure its closedness.
\begin{proposition}\label{weakclosed}
If $\underline{X}$ in $\mathfrak{X}$ is such that $\inf\{\delta_{*}(Y):Y\in \Gamma^{r^w}_{\underline{X}}(\mathcal{I}^{\mathbb{P}})\}>0$ then $\Gamma^{r^w}_{\underline{X}}(\mathcal{I}^{\mathbb{P}})$ is a closed set in $(\mathfrak{X^0},\rho)$.
\end{proposition}
\begin{proof}
Assume that $\{Y_n\}_{n\in\mathbb{N}}$ is a sequence in $\Gamma^{r^w}_{\underline{X}}(\mathcal{I}^{\mathbb{P}})$ such that $Y_n\xrightarrow{\rho} Z$. Then, by the given hypothesis, we have $\ell=\inf\{\delta_{*}(Y_n):n\in\mathbb{N}\}>0$. Let $0<\varepsilon<\frac{\ell}{2}$. Since $\rho$ metrizes convergence in probability, there exists $n_0\in\mathbb{N}$ such that
$$
\mathbb{P}(d(Y_{n_0},Z)\geq\frac{\varepsilon}{2})<\frac{\varepsilon}{2}.
$$
Since $Y_{n_0}\in\Gamma^{r^w}_{\underline{X}}(\mathcal{I}^{\mathbb{P}})$, we have
$
\{n\in\mathbb{N}:\mathbb{P}(d(X_n,Y_{n_0})<r+\frac{\varepsilon}{2})>\delta_{*}(Y_{n_0})\}\notin\mathcal{I}$. Therefore, we can write
$$
A(\varepsilon)=\{n\in\mathbb{N}:\mathbb{P}(d(X_n,Y_{n_0})<r+\frac{\varepsilon}{2})>\ell\}\notin\mathcal{I}.
$$
Pick arbitrary $n\in A(\varepsilon)$. Then, observe that
\begin{align*}
	\mathbb{P}(d(X_n,Z)\geq r+\varepsilon)&\leq  \mathbb{P}(d(X_n,Y_{n_0})\geq r+\frac{\varepsilon}{2})+	\mathbb{P}(d(Y_{n_0},Z)\geq \frac{\varepsilon}{2})\\
	&\leq \mathbb{P}(d(X_n,Y_{n_0})\geq r+\frac{\varepsilon}{2})+\frac{\varepsilon}{2}\\
	\Rightarrow 	\mathbb{P}(d(X_n,Z)<r+\varepsilon)&>\mathbb{P}(d(X_n,Y_{n_0})<r+\frac{\varepsilon}{2})-\frac{\varepsilon}{2}>\ell-\frac{\varepsilon}{2}>\frac{\ell}{2}.
\end{align*}
This entails that  $\{n\in\mathbb{N}:\mathbb{P}(d(X_n,Z)<r+\varepsilon)>\frac{\ell}{2}\}\supseteq A(\varepsilon)$. Since $A(\varepsilon)\notin\mathcal{I}$, we deduce that $\{n\in\mathbb{N}:\mathbb{P}(d(X_n,Z)<r+\varepsilon)>\frac{\ell}{2}\}\notin\mathcal{I}$, i.e., $Z\in\Gamma^{r^w}_{\underline{X}}(\mathcal{I}^{\mathbb{P}})$. Hence $\Gamma^{r^w}_{\underline{X}}(\mathcal{I}^{\mathbb{P}})$ is closed in $(\mathfrak{X^0},\rho)$.
\end{proof}
	Recall that an ideal $\mathcal{I}$ on $\mathbb{N}$ is termed maximal if, for any ideal $\mathcal{J}$ on $\mathbb{N}$ satisfying $\mathcal{I} \subseteq \mathcal{J} \subseteq \mathbb{N}$, we have either \( \mathcal{J} = \mathcal{I} \) or $ \mathcal{J} = \mathbb{N}$. The following fact regarding maximal ideals will be utilized in this paper.
	\begin{lemma}\cite[Lemma 5.1]{kostyrko2000convergence}\label{lemma maximal}
		Let $\mathcal{I}$ be an maximal admissible ideal on $\mathbb{N}$. Then for each $A\subset\mathbb{N}$, we have either $A\in\mathcal{I}$ or $\mathbb{N}\setminus A\in\mathcal{I}$.
	\end{lemma}
We are now in a position to present a pivotal result, namely Theorem \ref{maximal th}, which provides a characterization of maximal admissible ideals in terms of $\mathcal{I}^{\mathbb{P}}$-$LIM^rX_i$ and $\Gamma^{r^s}_{\underline{X}}(\mathcal{I}^{\mathbb{P}})$.
\begin{theorem}\label{maximal th}
Suppose $(\mathfrak{X}^0, \rho)$ has at least two distinct elements, and $\mathcal{I}$ is an admissible ideal. Then $\mathcal{I}$ is maximal if and only if $\mathcal{I}^{\mathbb{P}}$-$LIM^rX_i=\Gamma^{r^s}_{\underline{X}}(\mathcal{I}^{\mathbb{P}})$ holds for each $r\geq 0$ and each $\underline{X}$ in $\mathfrak{X}$.
\end{theorem}
\begin{proof}
	Assume that $\mathcal{I}$ is a maximal admissible ideal on $\mathbb{N}$. Let $\varepsilon,\delta>0$ be arbitrary.
Pick any $r\geq 0$ and any sequence $\underline{X}$ with values in $\mathfrak{X}$. Then, for each $X_*\in\mathcal{I}^{\mathbb{P}}$-$LIM^rX_i$, we have 
\begin{align*}
	&\{n\in\mathbb{N}:\mathbb{P}(d(X_n,X_*)\geq r+\varepsilon)>\delta\}\in\mathcal{I}\\
	\Rightarrow~& \{n\in\mathbb{N}:\mathbb{P}(d(X_n,X_*)\geq r+\varepsilon)\leq\delta\}\notin\mathcal{I}~\mbox{ (since $\mathcal{I}$ is non-trivial)}\\
	\Rightarrow~& \{n\in\mathbb{N}:1-\mathbb{P}(d(X_n,X_*)< r+\varepsilon)\leq\delta\}\notin\mathcal{I}\\
\Rightarrow~& \{n\in\mathbb{N}:\mathbb{P}(d(X_n,X_*)< r+\varepsilon)\geq 1-\delta\}\notin\mathcal{I}.
\end{align*}
This ensures that $X_*\in \Gamma^{r^s}_{\underline{X}}(\mathcal{I}^{\mathbb{P}})$, i.e., $\mathcal{I}^{\mathbb{P}}$-$LIM^rX_i\subseteq\Gamma^{r^s}_{\underline{X}}(\mathcal{I}^{\mathbb{P}})$.\\

On the other hand, pick any $Y\in\Gamma^{r^s}_{\underline{X}}(\mathcal{I}^{\mathbb{P}})$. Then observe that 
\begin{align*}
	&\{n\in\mathbb{N}:\mathbb{P}(d(X_n,X_*)< r+\varepsilon)\geq 1-\delta\}\notin\mathcal{I}\\
	\Rightarrow~& \{n\in\mathbb{N}:\mathbb{P}(d(X_n,X_*)< r+\varepsilon)<1-\delta\}\in\mathcal{I}~\mbox{ (since $\mathcal{I}$ is admissible and maximal)}\\
	\Rightarrow~& \{n\in\mathbb{N}:1-\mathbb{P}(d(X_n,X_*)\geq r+\varepsilon)<1-\delta\}\in\mathcal{I}\\
	\Rightarrow~& \{n\in\mathbb{N}:\mathbb{P}(d(X_n,X_*)\geq r+\varepsilon)>\delta\}\in\mathcal{I}.
\end{align*}
Consequently, we obtain $Y\in\mathcal{I}^{\mathbb{P}}$-$LIM^rX_i$, i.e., $\Gamma^{r^s}_{\underline{X}}(\mathcal{I}^{\mathbb{P}})\subseteq \mathcal{I}^{\mathbb{P}}$-$LIM^rX_i$. Hence we can deduce that $\mathcal{I}^{\mathbb{P}}$-$LIM^rX_i=\Gamma^{r^s}_{\underline{X}}(\mathcal{I}^{\mathbb{P}})$.\\

	Conversely, suppose that $\mathcal{I}^{\mathbb{P}}$-$LIM^rX_i=\Gamma^{r^s}_{\underline{X}}(\mathcal{I}^{\mathbb{P}})$ holds for each $r\geq 0$ and each $\underline{X}$ in $\mathfrak{X}$.
Assume, on the contrary, that $\mathcal{I}$ is not maximal. Then there exists $A\subset\mathbb{N}$ such that either $A,\mathbb{N}\setminus A\in\mathcal{I}$ or $A,\mathbb{N}\setminus A\notin\mathcal{I}$. Since $\mathcal{I}$ is non-trivial, we must have $A,\mathbb{N}\setminus A\notin\mathcal{I}$. By the given hypothesis, there exist $Y,Z\in \mathfrak{X}^0$ such that $\alpha=\rho(Y,Z)>0$. Since the infimum in the definition of $\rho(Y,Z)=\inf \{\varepsilon>0:\mathbb{P}(d(Y,Z)>\varepsilon)\leq\varepsilon\}$ is attained, we have 
$$
\mathbb{P}(d(Y,Z)>\alpha)\leq\alpha.
$$
We fix any $0<\beta<\alpha$. 
Now, pick any $\varepsilon>0$ such that $\beta+\varepsilon<\alpha$.
Then, we will have 
$$
\mathbb{P}(d(Y,Z)>\beta+\varepsilon)>\beta+\varepsilon~\mbox{ (using the minimality of $\alpha$)}.
$$
Let us define $\underline{X}$ in $\mathfrak{X}$ as follows:
\begin{align*}
	X_n=
	\begin{cases}
		Y~&\mbox{ if }~ n\in A,\\
		Z ~&\mbox{ if }~n\in \mathbb{N}\setminus A.
	\end{cases}
\end{align*}
Then, observe that 
$$
A\subseteq\{n\in\mathbb{N}:\mathbb{P}(d(X_n,Y)<\beta+\varepsilon)>1-\delta\}\notin\mathcal{I}~\mbox{ for every }~\delta>0,
$$
i.e., $Y\in\Gamma^{\beta^s}_{\underline{X}}(\mathcal{I}^{\mathbb{P}})$. Also, observe that 
$$
\mathbb{N}\setminus A\subseteq\{n\in\mathbb{N}:\mathbb{P}(d(X_n,Y)>\beta+\varepsilon)>\beta+\varepsilon\}\notin\mathcal{I},
$$
i.e., $Y\notin \mathcal{I}^{\mathbb{P}}$-$LIM^{\beta}X_i$ - which is a contradiction. Hence we deduce that $\mathcal{I}$ is maximal.
\end{proof}

%\section*{Acknowledgments} We are thankful to Prof. Pratulananda Das, Jadavpur University, West Bengal, India for several valuable suggestions which improved the quality and presentation of the paper.

\end{document}